\documentclass[12pt]{article}
\usepackage{dsfont}
\usepackage{multirow}
\usepackage{mathrsfs}
\usepackage{amsmath}
\usepackage{amsthm}
\usepackage{mathtools}
\usepackage{cases}
\usepackage{amssymb}
\usepackage{graphicx}
\usepackage{pstricks}
\usepackage{pstricks-add}
\usepackage{epic}
\usepackage{color}
  \usepackage{calc}
  \usepackage{tikz}
\usepackage[noblocks]{authblk}
\usepackage{bbm}

\newtheorem{definition}{Definition}[section]
\newtheorem{theorem}[definition]{Theorem}
\newtheorem{lemma}[definition]{Lemma}
\newtheorem{corollary}[definition]{Corollary}
\newtheorem{remark}[definition]{Remark}
\newtheorem{example}[definition]{Example}

\newtheorem{proposition}[definition]{Proposition}

\newcommand{\rk}{\rm Rk}

\def\R{\mathds R}

\def\Ga{\Gamma}

\begin{document}

\title{\bf Distance-regular graphs with exactly one positive $q$-distance eigenvalue}

\author[a,b]{Jack H. Koolen}

\author[a]{Mamoon Abdullah}

\author[a]{Brhane Gebremichel}

\author[c,$_{\footnote{Corresponding author}}$]{Sakander Hayat}

\affil[a]{\emph{School of Mathematical Sciences,
University of Science and Technology of China,
Hefei, Anhui, 230026, PR China.}}

\affil[b]{\emph{CAS Wu Wen-Tsun Key Laboratory of Mathematics,
University of Science and Technology of China,
Hefei, Anhui, 230026, PR China.}}

\affil[c]{\emph{Faculty of Science, Universiti Brunei Darussalam,
Jln Tungku Link, Gadong BE1410, Brunei Darussalam.}}

\date{May 24, 2023}
\maketitle
\newcommand\blfootnote[1]{%
\begingroup
\renewcommand\thefootnote{}\footnote{#1}%
\addtocounter{footnote}{-1}%
\endgroup}
\blfootnote{E-mail addresses: {\tt koolen@ustc.edu.cn} (J.H. Koolen), {\tt mamoonabdullah@hotmail.com} (M. Abdullah), 
{\tt brhane@ustc.edu.cn} (B. Gebremichel),  {\tt sakander1566@gmail.com} (S. Hayat)}
\begin{abstract}
In this paper, we study the $q$-distance matrix for a distance-regular graph 
and show that the $q$-distance matrix of a distance-regular graph with classical parameters 
$(D, q, \alpha, \beta)$ has exactly three distinct eigenvalues, 
of which one is zero. Moreover, we study distance-regular graphs whose 
$q$-distance matrix has exactly one positive eigenvalue.
\end{abstract}

\section{Introduction}
All graphs in this paper are connected, finite, simple, and undirected. For undefined notions and 
more information on graphs and their spectra, we refer to \cite{BH11, GD01}.

Now we will discuss the $q$-distance matrix of a graph. We will see that this matrix will occur naturally,
when we discuss distance-regular graphs with classical parameters.
Let $G$ be a connected graph. We denote by $d(x,y)$ the distance between vertices $x$ and $y$ of $G$. For 
a real number $q \neq 0$, we define the {\em $q$-distance matrix} $\mathbb{D}_q$ by
$(\mathbb{D}_q)_{xy} = 1 + \frac{1}{q} + \cdots + \frac{1}{q^{d(x,y) -1}}$ for $x \neq y$ and 0 otherwise, where $x$ and $y$ are vertices of $G$.

Note that in Bapat et al. \cite{BLP2006}, they defined the
$q$-distance matrix in a different way. If you replace $\frac{1}{q}$ by $q$ in our definition, you obtain their definition. For us, our
definition is the more natural definition. For more work on the $q$-distance matrix, we refer to 
\cite{Jana2023,LSZ2014,Siva2009,Siva2010,WIS2012}.
Note that, for $q=1$, one obtains the usual distance matrix.

For a distance-regular graph $\Gamma$ with classical parameters $(D, q, \alpha, \beta)$ (see Subsection \ref{subsec:classicalDRG} for a definition), 
we show that the $q$-distance matrix of $\Gamma$ has 
exactly three distinct eigenvalues of which one is equal to 0. This is an important property, 
see Theorems \ref{hier} and \ref{rowsum}. In order to do so, we will define an eigenvalue of classical $q$-type, 
and show that $\Gamma$ has an eigenvalue of classical $q$-type. This enables us to give alternative proofs of some 
results of Aalipour et al. \cite{AAB2016}. Also, we will show that if a distance-regular graph has an eigenvalue of 
classical $q$-type with $q >0$, then any local graph has smallest eigenvalue at least $-q-1$. 
As for given $q\geq 2$, it is a difficult problem to classify the distance-regular graphs with classical 
parameters $(D, q, \alpha, \beta)$ (see Subsection \ref{subsec:classicalDRG} for an explanation), the problem of 
classifying distance-regular graphs with an eigenvalue of classical $q$-type is even more difficult. The main motivation 
for this paper is that it provides an alternative way to approach this problem. 

This paper is organized as follows: In Section 2, we give preliminaries regarding graphs, 
non-negative matrices and the metric hierarchy. In Section 3, we discuss the basic theory of 
distance-regular graphs and we define distance-regular graphs with classical parameters and some of their properties. 
In Section 4, we discuss distance-regular graphs with an 
eigenvalue of classical $q$-type and show that the corresponding $q$-distance matrix has exactly three distinct 
eigenvalues of which one is equal to 0. In the last section, we study 
distance-regular graphs whose $q$-distance matrix has exactly one positive eigenvalue. 

\section{Preliminaries}
\subsection{Definitions}
A \emph{graph} $G$ is a pair $(V,E)$, where $V$ is called the vertex set and $E(G)\subseteq{V(G)\choose2}$
is called the edge set of $G$. Denoted by $u\sim v$, two vertices $u,v\in V(G)$ are said to be \emph{adjacent}
if $uv\in E(G)$. The valency $d_u$ of a vertex $u\in V(G)$ is defined as $d_u=|\{v\in V(G):uv\in E(G)\}|$.
A graph is said to be \emph{regular} of valency $k$, if $d_u=k$ for each $u\in V(G)$. 
A graph is called
\emph{bipartite} if the vertex set can be partitioned into two parts such that every edge has one
end (vertex) in each part.

A graph is said
to \emph{coconnected} if its complement is connected. An $n$-vertex $k$-regular graph is said
to be \emph{strongly-regular} with parameters $(n,k,a,c)$ if every pair of adjacent (resp. nonadjacent)
vertices has exactly $a$ (resp. $c$) common neighbors.
A connected and coconnected strongly-regular graph is called \emph{primitive}, otherwise it
is called \emph{imprimitive}.
The \emph{distance} $d(u,v)$ between two vertices $u$ and $v$ is the length of a shortest path
between them. The \emph{distance-}$i$ \emph{graph} $G_i$ is the graph with vertex set
$V$, where two vertices $x$ and $y$ are adjacent if and only if $d_G(x,y) = i$. 

The \emph{Kronecker product} $M_1\otimes M_2$ of two matrices $M_1$ and $M_2$
is obtained by replacing the $ij$-entry of $M_1$ by $(M_1)_{i,j}M_{2}$ for all $i$ and $j$.
Note that if $\tau$ and $\eta$ are eigenvalues of $M_1$ and $M_2$ respectively, then $\tau\eta$
is an eigenvalue of $M_1\otimes M_2$.
For a positive integer $s$, the \emph{$s$-clique extension} of $G$ is the graph $\widetilde{G}$ obtained from
$G$ by replacing each vertex $x\in V(G)$ by a clique $\widetilde{X}$ with $s$ vertices, such that $\tilde{x}\sim\tilde{y}$
(for $\tilde{x}\in \widetilde{X},~\tilde{y}\in \widetilde{Y}$) in $\widetilde{G}$ if and only if
$x\sim y$ in $G$. If $\widetilde{G}$ is the $s$-clique extension of $G$, then $\widetilde{G}$ has adjacency matrix
$(A+I_n)\otimes J_{s}-I_{sn}$,
where $J_s$ is the all-ones matrix of size $s$ and $I_n$ is the identity matrix of size $n$. In particular, if $G$ has spectrum
\begin{equation*}\label{G-spectrum}
\{\theta_0^{m_0},\theta_1^{m_1},\ldots,\theta_t^{m_t}\},
\end{equation*}
then it follows that the spectrum of $\widetilde{G}$ is
\begin{equation*}\label{sclique-spectrum}
\big\{\big(s(\theta_0+1)-1\big)^{m_0}, \big(s(\theta_1+1)-1\big)^{m_1},\ldots,\big(s(\theta_t+1)-1\big)^{m_t},(-1)^{(s-1)(m_0+m_1+\ldots+m_t)}\big\}.
\end{equation*}

\subsection{Matrices}
We denote the eigenvalues of a real symmetric matrix $M$ of order $n$
by $\eta_1(M)\geqslant\eta_2(M)\geqslant\cdots\geqslant\eta_n(M)$. The largest (resp. smallest)
eigenvalue of $M$ is also denoted by $\rho(M)$ (resp. $\eta_{\min}(M)$). The
{\em spectral radius} $\rho(M)$ of a matrix $M$ is the maximum of the moduli of its eigenvalues.
The {\em rank} of $M$ is denoted by $\rk(M)$. Let $J$ (resp. $I$) be the all-ones
(resp. identity) matrix of a suitable order.

For a real symmetric $n\times n$ matrix $B$ and a real symmetric $m\times m$ matrix $C$ with $n>m$,
we say that the eigenvalues of $C$ {\em interlace} the eigenvalues of $B$, if
$$\eta_{n-m+i}(B)\leqslant\eta_i(C)\leqslant\eta_i(B)$$
for each $i=1,\ldots,m$. The following result is a special case of interlacing.

\begin{theorem}[Cf. {\cite[Theorem 9.1.1]{GD01}}]\label{interlacing}
Let $B$ be a real symmetric $(n\times n)$-matrix and $C$ be a principal submatrix of $B$ of order $m$, where $m<n$.
Then, the eigenvalues of $C$ interlace the eigenvalues of $B$.
\end{theorem}

A real $(n \times n)$-matrix $T$ with only non-negative entries,
is called \emph{irreducible} if for all $1 \leq i, j \leq n$,
there exists $k >0$ such that $(T^k)_{ij} >0$.

Now we introduce an important consequence of the Perron-Frobenius Theorem.
\begin{theorem}[Cf. {\cite[Theorem 3.1.1]{BCN}}] \label{PF}
Let $T$ be a nonnegative irreducible matrix. Let $\rho=\rho(T):= \max \{|\theta|:\theta \text{ is an eigenvalue of } T\}$.
Then, $\rho$ is an eigenvalue of $T$ with (algebraic) multiplicity one, and, if $\mathbf{x}$ is an eigenvector for $\rho$, 
then all entries of $\mathbf{x}$ are nonzero and have the same sign.
Moreover, if $\theta$ is an eigenvalue of $T$ with an eigenvector having only positive real numbers, then $\theta = \rho$.
\end{theorem}

As a consequence, we obtain the following result.

\begin{proposition} \label{3evnew}
Let $G$ be a connected graph with $n \geq 3$ vertices and with adjacency matrix $A$.
Let $M$ be a symmetric non-negative irreducible $(n \times n)$-matrix indexed by the
vertex set of $G$ with $M_{xy} = 0$ if $x=y$ or $x \sim y$.   Then,
$B=M+A$ has at least three distinct eigenvalues unless
$B = J - I$. Note that the spectral radius has multiplicity 1.
\end{proposition}
\begin{proof} Assume that $B \neq J - I$.
Let $x,y,z$ be three distinct vertices such that $B_{xy} = 1 = B_{xz} \neq B_{yz}$.
Then, the prinicipal submatrix of $B$ indexed by $\{x, y, z\}$ has three distinct eigenvalues
$\theta_0 > \theta_1 > \theta_2$ such that $\theta_1 > -1> \theta_2$. By the fact that the
spectral radius has multiplicity 1 (by Theorem \ref{PF}), the smallest eigenvalue of $B$ is
less than $-1$ and $B$ has at least two eigenvalues  at least $-1$ (by Theorem \ref{interlacing}),
we find that $B$ has at least three distinct eigenvalues. This shows the proposition.
\end{proof}

\subsection{Metric spaces and the metric hierarchy}
In this subsection, we study metric spaces and the metric hierarchy. For more information on this, see \cite{DeLa97}. 
Let $X$ be a nonempty set and $d: X \times X \rightarrow \mathbb{R}_{\geq 0}$. 
We say that $(X, d)$ is a {\em distance space} if $d(x, y)= d(y, x)$ and $d(x,x) = 0$ for all $x,y \in X$.  
It is called a {\em semi-metric space} (or a {\em semi-metric})  if the triangle inequality is satisfied, i.e. $d(x, y) + d(y,z) \geq d(x, z)$ for all $x, y, z \in X$. 
A semi-metric space is called {\em metric} if $d(x, y) = 0$ implies $x=y$ for all $x, y \in X$. 
If $X$ is finite, then we define the {\em distance matrix} of a distance space $(X, d)$ as the symmetric matrix $\mathbb{D}= \mathbb{D}(X,d)$ whose rows and columns are indexed by $X$ such that $\mathbb{D}_{xy} = d(x,y)$. 

Let $(X, d)$ and $(X', d')$ be two distance spaces. Then $(X,d)$ is said to be {\em isometrically embeddable} into $(X', d')$ if there exists a map $\sigma: X \rightarrow X'$ such that 
$d(x,y) = d'(\sigma(x), \sigma(y))$ for all $x, y \in X$. We call $\sigma$ the {\em isometric embedding} of $(X,d)$ into $(X', d')$.  
We also say that $(X,d)$ is an {\em isometric subspace} of $(X', d')$. 

Let $m$ be a positive integer. Let $1 \leq p $ be a real number. We define the distance $\ell_p$ on $\mathbb{R}^m\times \mathbb{R}^m$ by 
$$\ell_p(\mathbf{x}, \mathbf{y}) = \left(\sum_{i=1}^m |x_i - y_i|^p\right)^{\frac{1}{p}}.$$

For $p \geq 1$, a distance space $(X, d)$ is {\em $\ell_p$-embeddable} if there exists a positive integer $m$ such that $(X, d)$ is isometrically embeddable in $(\mathbb{R}^m, \ell_p)$.

We are mainly interested in $p=1$ and $p=2$ in this paper. 

Let $(X, d)$ be a distance space. Let $b: X \rightarrow \mathbb{Z}$ be a map. 
Consider the inequality 
\begin{equation}\label{Ineq}
\sum_{x \in X}\sum_{y \in X} b_x b_y d(x, y) \leq 0. 
\end{equation}

The inequalities $d(x,y) \geq 0$ and $d(x, y) +d(x, z) \geq d(y,z)$ can be obtained by setting $b: X \rightarrow \mathbb{Z}$ with $b_x =1, b_y =-1, b_z =0$ for $z \in X-\{x, y\}$ respectively $b_x =-1, b_y =1, b_z =1$, and $b_u=0$ for $u \in X-\{x, y, z\}$ in Inequality (\ref{Ineq}).

We say that the distance space $(X, d)$ is {\em hypermetric}, respectively {\em of negative type}, if Inequality (\ref{Ineq}) is satisfied for all maps 
$b: X \rightarrow \mathbb{Z}$ satisfying
$\sum_{x \in X} b_x =1$, respectively $\sum_{x \in X} b_x= 0$. 

We have the following  hierarchy of distance spaces, see \cite[Theorem 6.3.1]{DeLa97}.

\begin{theorem}{\em (Metric Hierarchy)}\label{hier}
Let $(X, d)$ be a finite distance space with associated distance matrix $\mathbb{D} \neq \mathbf{0}$. Consider the following conditions:
\begin{enumerate}
\item[$(i)$] $(X,d)$ is $\ell_2$-embeddable;
\item[$(ii)$] $(X,d)$ is $\ell_1$-embeddable;
\item[(iii)] $(X,d)$ is hypermetric;
\item[(iv)] $(X,d)$ is of negative type;
\item[(v)] $(X,\sqrt{d})$ is $\ell_2$-embeddable;
\item[(vi)] $\mathbb{D}$ has exactly one positive eigenvalue.
\end{enumerate}
Then we have the following chain of implications:
$(i) \implies (ii) \implies (iii) \implies (iv) \Leftrightarrow (v) \implies (vi)$
\end{theorem}

For finite distance spaces with distance matrices having constant row sum, the next result shows that  several of the implications of Theorem \ref{hier} collapse.

\begin{theorem}{\em(\cite[Theorem 6.2.18]{DeLa97})}\label{rowsum}
Let $(X, d)$ be a finite distance space with  $|X|=n$ and with associated distance matrix $\mathbb{D} \neq \mathbf{0}$. Assume that $\mathbb{D}$ has constant row sum $\theta$.
 Then the following conditions are 
equivalent.
\begin{enumerate}
\item[(i)] $(X,d)$ is of negative type;
\item[(ii)] $(X,\sqrt{d})$ is $\ell_2$-embeddable;
\item[(iii)] There exists vectors $\mathbf{u}^{(x)}$ for $x \in X$ such that $(\mathbf{u}^{(x)}, \mathbf{u}^{(x)}) = \frac{\theta}{2n}$ and 
$d(x,y) = (\mathbf{u}^{(x)}-\mathbf{u}^{(y)}, \mathbf{u}^{(x)}-\mathbf{u}^{(y)})$ for all $x, y \in X$;
\item[(iv)] $\mathbb{D}$ has exactly one positive eigenvalue.
\end{enumerate}

\end{theorem}

\subsection{Generalized distance matrices and the $q$-distance matrix}

We first introduce the notion of a generalized distance matrix of a connected graph.
Our notion is slightly less general than the notion of a generalized distance matrix as defined by \cite{DeVille2022}.
Let $G$ be a connected graph with diameter $D$. Let $\mathbf{\alpha} = (\alpha_0, \alpha_1, \ldots, \alpha_D)$
be a sequence of real numbers such that $\alpha_0 = 0$ and $\alpha_1 =1$.
Then, the symmetric matrix $\mathbb{D} = \mathbb{D}_{\mathbf{\alpha}}$ indexed by the vertex set of $G$
such that $\mathbb{D}_{xy} = \alpha_{d(x,y)}$ is called a {\em
generalized distance matrix}. We allow that some of the $\alpha_i$'s are negative. We say that $(V(G), f)$ is a {\em generalized distance space} where 
$f(x, y) = \alpha_{d(x,y)}.$

We will see that for a distance-regular graph $\Gamma$ and
a non-trivial eigenvalue $\theta$ of $\Gamma$, there exists a natural generalized distance matrix $\mathbb{D}(\theta)$.

Now we have the following lemma for the $q$-distance matrix.
\begin{lemma}\label{nonnegative}
Let $G$ be a connected non-complete graph. Let $q \neq 0$ be a real number.
Then,
\begin{itemize}
\item[(i)] The matrix $\mathbb{D}_q$ of $G$  is a non-negative irreducible matrix if and only if $q > 0$ or $q\leq -1$.
\item[(ii)] The matrix $\mathbb{D}_q$ of $G$ is the distance matrix of a metric space if $q \geq 1$ or $q \leq -2$.
\end{itemize}
\end{lemma}
\begin{proof}
Note that $\mathbb{D}_q$ is non-negative if $q >0$ or $q \leq -1$ is straightforward. If $-1 < q < 0$, 
then as the graph is connected we have negative entries in the matrix
and the matrix $\mathbb{D}_q$ is then also irreducible. 

Now assume $q \geq 1$. Let $x, y, z$ be three vertices of $G$ say with
$d(x, y)=i,~d(y, z) = j$ and $d(x, z) = k$. If at least one of $i,~j$ or $k$
is zero, we can conclude that the triangle inequality holds, as long the
matrix $\mathbb{D}_q$ is non-negative. As $q\geq1$ and $k \leq i +j$, we see that
$ 1 + \frac{1}{q} + \cdots + \frac{1}{q^{i-1}} + 1 + \frac{1}{q} + \cdots + \frac{1}{q^{j-1}}
\geqslant 1 + \frac{1}{q} + \cdots + \frac{1}{q^{k-1}}$ holds, if $i,~j$ and $k$ are all positive.

Now assume $ q \leq -2$. Then, the off-diagonal entries of $\mathbb{D}_q$ lie between $1$ and
$1/2$. It follows that the triangle inequality is always satisfied.
\end{proof}

Note that if $G$ is a connected bipartite graph with at least $3$ vertices, then the generalized distance space of $\mathbb{D}_{-1}$ is a semi-metric, so it satisfies the triangle inequality. 

As a consequence of Lemma \ref{nonnegative}, we have the following result.
\begin{lemma}\label{3qev}
Let $G$ be a connected non-complete graph. Let $q$ be a real number such that $q > 0$ or $q \leq -1$. 
Then, the $q$-distance matrix $\mathbb{D}_q$ of $G$
has at least 3 distinct eigenvalues, of which, the spectral radius $\rho$ has multiplicity 1.
\end{lemma}
\begin{proof}
Let $q >0$ or $q \leq -1$. Then, the matrix $\mathbb{D}_q$ of $G$ is non-negative and irreducible,
so its spectral radius $\rho$ has multiplicity one and $\rho$ has an eigenvector with only positive
entries, by Theorem \ref{PF}. Now consider the path $P$ of length 2. Then, the $q$-distance
matrix $\mathbb{D}'_q$ of $P$ has three eigenvalues $\theta_0 > \theta_1 \geq \theta_2$.

By considering an edge as an induced subgraph of $P$, we see that $\theta_1 \geq -1 \geq \theta_2$,
by Theorem \ref{interlacing}. It is straightforward to see that
$\mathbb{D}_q'$
does not have $-1$ as an eigenvalue. So this means, again by interlacing, that $\mathbb{D}_q$ has
at least two eigenvalues more than $-1$ and smallest eigenvalue less that $-1$.
As the spectral radius $\rho$ has multiplicity 1, this shows the lemma.
\end{proof}

\begin{example}
Consider the complete bipartite graph $K_{r,s}$ with $r \geq s \geq 1$. Then the $-\frac{1}{2}$-distance matrix 
$\mathbb{D}_{-\frac{1}{2}}$ of $K_{r,s}$ has distinct eigenvalues $-r-s+1$
and $1$ with respective multiplicities $1$ and $r+s-1$.
\end{example}

This example can be extended to all connected non-complete strongly-regular graphs. Also, for the Johnson graph $J(6,3)$, the $-\frac{1}{2}$-distance matrix $\mathbb{D}_{-\frac{1}{2}}$ has exactly two distinct eigenvalues.  In a follow-up paper, we will study distance-regular graphs with a few distinct $q$-distance eigenvalues.

\section{Distance-regular graphs}
In this section, we discuss distance-regular graphs. For more information on them, see \cite{BCN, DaKoTa}.
\subsection{Distance-regular graphs and the Bose-Mesner algebra}

A connected graph $\Gamma$ of diameter $D$ is said to be {\em distance-regular} if and only if
for all integers $h, i, j$ with $0 \leq h, i, j \leq D$ and all
vertices $x, y \in V(\Gamma)$ with $d(x, y) = h$, the number
$$p_{ij}^h := |\{z\in V(\Gamma) \mid d(x, z) = i, ~d(y, z) = j\}| = |\Gamma_i(x)\cap \Gamma_j(y)|$$
is independent on the choice of $x$ and $y$. The constants $p^h_{ij}$
are called the {\em intersection
numbers} of $\Gamma$. We abbreviate $c_i = p^i_{1i-1}~(1\leq i \leq D)$,
$a_i = p^i_{1i}~(0 \leq i \leq D)$, and $b_{i} = p^{i}_{1i+1}~(0 \leq i \leq D-1)$.
Observe that $\Gamma$ is regular with valency $k = b_0$, and
$c_i + a_i + b_i = k$ for $0 \leq i \leq D$, where we define $c_0=b_{D}=0$.
The array $\{b_0, b_1,\ldots , b_{D-1}; c_1, c_2,\ldots , c_D\}$ is called
the {\em intersection array} of the distance-regular graph $\Gamma$.

Let $\Gamma$ be a distance-regular graph of diameter $D$ and $n$ vertices.
For each integer $i$ with $0 \leq i \leq D$, define the {\em $i^{\rm{th}}$
distance matrix} $A_i$ of $\Gamma$ whose rows and columns are indexed by the vertices of $\Gamma$, by
\begin{equation*}
(A_{i})_{xy}=
\begin{cases}
1& \text{if} \enskip d(x,y)=i,\\
0 & \text{if} \enskip d(x,y)\neq i,
\end{cases} \quad\left(x, y \in V(\Gamma)\right).
\end{equation*}
Then, $A:=A_1$ is the {adjacency matrix} of $\Gamma$. Observe that $A_0 = I;~ A_i^\top = A_i ~(0 \leq i\leq D);~\sum_{i=0}^DA_i=J$; and
\begin{align*}A_iA_j=\sum\limits_{h=0}^D p_{ij}^h A_h\quad (0 \leq i, j\leq D).\end{align*}
It follows that $A_i$'s satisfy the following relations
\begin{equation}\label{recurrence relation}
  A_0=I, \text{ } A_1=A, \text{ } AA_i = c_{i+1}A_{i+1}  +a_{i}A_{i} +b_{i-1}A_{i-1}, \text{ for } i= 0, 1, 2, \ldots , D
\end{equation}
where $A_{-1} =A_{D+1} = \mathbf{0}$, while the numbers $b_{-1}$ and $c_{D +1}$ are unspecified.
 By these facts, we find that $\{A_0, A_1,\ldots,A_D\}$ is a basis
for a commutative subalgebra $\mathcal{M}$ of the matrix algebra over $\R$. We call $\mathcal{M}$ the {\em  Bose-Mesner algebra} of $\Gamma$.
It is known that $A$ generates  $\mathcal{M}$.
Since the algebra $\mathcal{M}$ is semi-simple and commutative, $\mathcal{M}$ also has a basis of
pairwise orthogonal idempotents $E_0,~E_1, \ldots,~E_D $ (the so-called {\em primitive idempotents} of $\mathcal{M}$)
satisfying
\begin{align*}
E_0 = \frac{1}{|V(\Gamma)|}J,\quad \sum_{i=0}^D E_i= I,\quad
  E_i^\top = E_i, \quad
E_iE_j = \delta_{ij}E_i \quad (0 \leq i, j \leq D),
\end{align*}
where $\delta_{ij}$ is the Kronecker delta.
Since $\mathcal M$ has two bases $\{A_i\}_{i=0}^D$ and $\{E_i\}_{i=0}^D$, there are real scalars $\{\theta_j\}_{j=0}^D$ such that
\begin{align*}
A=\sum_{j=0}^D \theta_j E_j.
\end{align*}
Observe that, $\theta_j,~0 \leq j \leq D$ is exactly the distinct eigenvalues of $A$ (of $\Gamma$), since $AE_j=E_jA=\theta_jE_j$. 
We say that, $E_j$ is the primitive idempotent corresponding to $\theta_j$.

Let $\theta$ be an eigenvalue of $\Gamma$. The standard sequence $(u_i)_{i=0}^D$ of $\theta$ satisfies
$u_0 =1$, $u_1 = \frac{\theta}{k}$, $ c_i u_{i-1} + a_i u_i + b_iu_{i+1} = \theta u_i$ for $i =1, 2, \ldots, D-1$. Note that $u_i$ are polynomials of degree $i$ in $\theta$.
It is known that the  primitive idempotent $E$ corresponding to $\theta$ satisfies $$E = \frac{\rk(E)}{n}\sum_{i=0}^D u_i A_i,$$ where $\rk(E)$ is the rank of $E$, which is also the multiplicity mult$(\theta)$ of the eigenvalue $\theta$ of $\Gamma$.
This implies that, there exists  a map $\sigma: V(\Gamma) \rightarrow \mathbb{R}^m$, where $m =$ mult$(\theta)$ is the multiplicity of $\theta$ as an eigenvalue of $\Gamma$ satisfying:
For all $x, y \in V(\Gamma)$, the inner product  $(\sigma(x), \sigma(y))= u_d$ where $d= d(x, y)$ is the distance between $x$ and $y$. 
This map $\sigma$ is called the {\em standard
representation} of $\Gamma$ with respect to $\theta$.

\subsection{Distance-regular graphs with classical parameters}\label{subsec:classicalDRG}
Recall that the {\em $q$-ary Gaussian binomial coefficient} is defined by
\begin{align}\label{Gaussiancoeff}
\Big[\substack{n\\ \\ m}\Big]_q=\frac{(q^n-1)(q^{n-1}-1)\cdots(q^{n-m+1}-1)}{(q^m-1)(q^{m-1}-1)\cdots(q-1)}.
\end{align}

\noindent We say that a distance-regular graph $\Gamma$ of diameter $D$ has {\em classical parameters} $(D, b, \alpha, \beta)$ if the intersection numbers of $\Gamma$ satisfy
\begin{align}\label{classical-ci}
c_i=\Big[\substack{i\\ \\ 1}\Big]_b\left(1+\alpha\Big[\substack{i-1\\ \\ 1}\Big]_b\right),
\end{align}
\begin{align}\label{classical-bi}
b_i=\left(\Big[\substack{D\\ \\ 1}\Big]_b-\Big[\substack{i\\ \\ 1}\Big]_b\right)\left(\beta-\alpha\Big[\substack{i\\ \\ 1}\Big]_b\right),
\end{align}
where $\Big[\substack{j\\ \\ 1}\Big]_b=1+b+b^2+\cdots b^{j-1}$ for $j\geq 1$ and $\Big[\substack{0\\ \\ 1}\Big]_b=0$. We notice that $b\neq 0,-1$ by the following result.
\begin{lemma} [{\cite[Proposition 6.2.1]{BCN}}]
Let $\Gamma$ be a distance-regular graph with classical parameters $(D, b, \alpha, \beta)$ and the diameter $D \geq 3$. 
Then, $b$ is an integer such that $b\neq 0, -1$.
\end{lemma}

There are many examples of distance-regular graphs with classical parameters such as the Johnson graphs, Hamming graphs, Grassmann graphs, bilinear forms graphs, and so on.
For more information on them, see for example \cite[Chapter 9]{BCN}.

A big open problem is to classify them for a large diameter, an important subproblem of Bannai's problem.
Another open problem is to characterize the known examples by their intersection numbers.
This is a difficult problem as, for example, the twisted Grassmann graphs $\tilde{J}_q(2D+1, D)$, as found by Van Dam and Koolen \cite{DK2005}, have the same intersection numbers as the Grassmann graphs $J_q(2D+1, D)$.

\section{Generalized distance matrices for distance-regular graphs}

Let $\Gamma$ be a distance-regular graph with diameter $D \geq 2$ and distinct eigenvalues $k=\theta_0 > \theta_1 > \cdots > \theta_D$.
Let $\mathbf{\alpha} = (\alpha_0, \alpha_1, \ldots, \alpha_D)$ be a sequence of real numbers such that  $\alpha_0=0$ and $\alpha_1 =1$.
Consider the generalized distance matrix $\mathbb{D} = \mathbb{D}_{\mathbf{\alpha}}$ defined by
$\mathbb{D}= \sum_{i=1}^D \alpha_i A_i$.
Then, from Equation (\ref{recurrence relation}), we obtain that the matrices $A_i$ are
polynomials of degree $i$ in $A$. As a consequence we find that, the matrix $\mathbb{D}$ is a
polynomial $R$ of degree at most $D$ in $A$. The eigenvalues of $\mathbb{D}$ are $R(\theta_i)$ for $i =0, 1, \ldots, D$.

There is another way to compute the eigenvalues of $\mathbb{D}$ for $\Gamma$. This was also
observed by DeVille \cite{DeVille2022}. Let $\theta$ be an eigenvalue of $\Gamma$ with standard
sequence $(u_i)_i$. Fix a vertex $x$ of $\Gamma$. Let $\mathbf{u}$ be a vector whose $y$-entry
$\mathbf{u}_y$ is equal to $u_{d(x, y)}$. Then, $A \mathbf{u} = \theta \mathbf{u}$ and,
hence, $\mathbf{u}$ is an eigenvector for $\mathbb{D}$ with eigenvalue
$\sum_{i=1}^D \alpha_ik_i u_i,$ as $$(\mathbb{D} \mathbf{u})_x= \sum_{i=1}^D \alpha_ik_i u_i = \sum_{i=1}^D \alpha_ik_i u_i \mathbf{u}_x,$$ as $\mathbf{u}_x = u_0  = 1.$

We also observe that $\mathbb{D} = \sum_{i=0}^D \eta_i E_i$ where $\eta_i = R(\theta_i)$ and $E_i$ is the $i^{\rm{th}}$ primitive idempotent of $\Gamma$.

\subsection{Standard generalized distance matrices}

Let $\Gamma$ be a distance-regular graph with diameter $D \geq 2$ and distinct eigenvalues $k=\theta_0 > \theta_1 > \cdots > \theta_D$.
Let $\theta\neq k$ be an eigenvalue of $\Gamma$ with standard sequence $(u_i)_i$.
The {\em standard generalized distance matrix} of $\Gamma$ is the generalized distance matrix $\mathbb{D}(\theta)$ defined as
$\mathbb{D}(\theta) = \sum_{i=1}^D \alpha_i A_i$, where $\alpha_i = \frac{1-u_i}{1-u_1}$ for $i =1, 2, \ldots, D$. As $\theta \neq k$ the numbers $\alpha_i$ are well-defined.
Moreover, $\alpha_i \geq 0$ for $i \geq 2$ and $\alpha_1 =1$. So, $\mathbb{D}(\theta)$ is an irreducible non-negative matrix.
We also have that $\mathbb{D}(\theta) = \gamma J + \delta E$ for some real numbers $\gamma, \delta$, where $E$ the primitive idempotent corresponding to $\theta$. This means that $\mathbb{D}(\theta)$ has exactly three distinct eigenvalues one of which is zero.

\subsection{Eigenvalues of classical type}
In this section, we will generalize the class of distance-regular graphs
with classical parameters. To this end, we introduce eigenvalues of classical type.

Let $\Gamma$ be a distance-regular graph with diameter $D$ at least 2 with valency $k$.
Let $q \neq 0$. We say that an eigenvalue $\theta\neq k$ of $\Gamma$ is of classical type (with parameter $q$) if there exists a real number $c$ such that
the standard sequence $(u_i)_i$ of $\theta$ satisfies
$u_i = \frac{1}{q} u_{i-1} +c$ for $i =1, 2, \ldots, D$.

In this case, we also say that $\theta$ is of classical $q$-type.
\begin{lemma}
Let $\Gamma$ be a distance-regular graph with diameter $D$ at least 2 with valency $k$. Let $\theta\neq k$ be an 
eigenvalue of classical $q$-type, where $q \neq 0$ is a real number. Let $(u_i)_i$ be the standard sequence with respect to $\theta$. 
Then, $\mathbb{D}(\theta) = \mathbb{D}_q$ and 
$u_i = 1 + \frac{\theta-k}{k}\left(1+ \frac{1}{q} + \cdots + \frac{1}{q^{i-1}}\right)$ for $i \geq 1$.
\end{lemma}
\begin{proof}
Let $\alpha_i = \frac{1-u_i}{1-u_1}$ for $i =0,1, \ldots, D$.
Then, $\alpha_0=0$, $\alpha_1 = 1$ and $\alpha_i = \frac{\alpha_{i-1}}{q} +1$ for $i =2, 3, \ldots, D$.
This shows that $\mathbb{D}(\theta) = \mathbb{D}_q$. This means $1- u_i =\left(1+ \frac{1}{q} + \cdots + \frac{1}{q^{i-1}}\right)(1-u_1)$, and so $$u_i = 1+(u_1-1)\left(1+ \frac{1}{q} + \cdots + \frac{1}{q^{i-1}}\right) = 1 + \frac{\theta-k}{k}\left(1+ \frac{1}{q} + \cdots + \frac{1}{q^{i-1}}\right).$$ This shows the lemma.
\end{proof}

\begin{remark}
Let $\Gamma$ be a distance-regular graph with diameter $D$ at least 2 with valency $k$.Let $\theta\neq k$ be an eigenvalue of classical $q$-type, where $q \neq 0$ is a real number.
Let $(u_i)_i$ be the standard sequence with respect to $\theta$. Let $\sigma: V(\Gamma) \rightarrow \mathbb{R}^m$ be the standard representation of $\Gamma$ with respect to 
$\theta$. Then, $(\sigma(x) - \sigma(y), \sigma(x) -\sigma(y)) = 2 - 2u_{d(x,y)}= \frac{2(k-\theta)}{k}\left(1+ \frac{1}{q} + \cdots + \frac{1}{q^{i-1}}\right)$, so 
this gives an explicit representation for Theorem \ref{rowsum} (iii). 
\end{remark}

We now obtain the following result.

\begin{proposition}\label{span}
Let $\Gamma$ be a distance-regular graph with diameter at least $2$ with valency $k$ and having $n$ vertices.
Let $\theta \neq k$ be an eigenvalue of $\Gamma$ with corresponding primitive idempotent $E$ and standard sequence $(u_i)_i$.
Let $q\neq 0$ be a real number.
Then, $\theta$ is of classical $q$-type if and only if the vector space $W$ spanned by the all-ones matrix $J$ and $E$ contains $\mathbb{D}_q$.
\end{proposition}
\begin{proof}
Assume first that $\theta$ is of classical $q$-type.
Then, the standard sequence $(u_i)_i$ of $\theta$ satisfies $u_i = \frac{1}{q} u_{i-1} + c$
for some real $c$ for $i=1, 2, \ldots, D$, and $u_0 = 1 > u_1$.
Also, $F = \sum_{i=0}^D u_i A_i$ is in the vector space $W$.
Let $\sigma_i := \frac{1-u_i}{1-u_1}$ for $i=0, 1, \ldots, D$. Then, $\sigma_0 = 0$, $\sigma_1 = 1$
and $\sigma_i = \frac{1}{q}\sigma_{i-1} +1$ for $i =1,2, \ldots, D$.
We obtain $\sigma_i = 1 +\frac{1}{q} + \cdots + \frac{1}{q^{i-1}}$ for $i =1,2, \ldots, D$.
This means that $\mathbb{D}_q = \sum_{i=0}^D \sigma_i A_i$ is in $W$.

Now assume that $W$ contains $\mathbb{D}_q$. Then, $\mathbb{D}_q = \sum_{i=0}^D \sigma_i A_i$,
where the numbers $\sigma_i$ are as above.
Note that $\sigma_i = \frac{1}{q} \sigma_{i-1} +1$. As the $u_i = \gamma \sigma_i + \delta$
for some real numbers $\gamma \neq 0$ and $\delta$, we obtain that
$\theta$ is of classical $q$-type.
\end{proof}

In the next result, we show that if $\theta$ is an eigenvalue of classical $q$-type,
then the $q$-distance matrix has exactly three distinct eigenvalues of which one is zero.
\begin{corollary}\label{3ev}
Let $\Gamma$ be a distance-regular graph with diameter at least $2$ with valency $k$ and having $n \geq 3$ vertices.
Let $\theta \neq k$ be an eigenvalue of $\Gamma$ with corresponding primitive idempotent $E$ and standard sequence $(u_i)_i$.
Let $q\neq 0$ be a real number.
Assume  $\theta$ is of classical $q$-type. Then, the $q$-distance matrix $\mathbb{D}_q$ has exactly three distinct eigenvalues of which one is zero.
Moreover $q >0$ or $q\leq -1$ holds.
\end{corollary}
\begin{proof}
Let $\theta$ be an eigenvalue of classical $q$-type. Let $F$ be a primitive idempotent of $\Gamma$. Then, $EF = E$ if $E=F$ and 0 otherwise. 
Similarly, $JF =0$ if $F \neq \frac{1}{n}J$.
As $\mathbb{D}_q$ the subspace generated by $E$ and $J$, by Proposition \ref{span}, we find that 
$\mathbb{D}_q F \neq 0$ if $F \not \in \{ \frac{1}{n}J, E\}.$
Because of  $\mathbb{D}_qE = \eta E$ and $\mathbb{D}_qJ= \sigma J$ for some real numbers $\eta$ and $\sigma$, we find that $\mathbb{D}_q$ has at most three distinct eigenvalues.
As $\mathbb{D}_q \neq 0$ and $n \geq 3$, we find that $\mathbb{D}_q$ must have a positive and a negative eigenvalue.
Note that $\mathbb{D}(\theta) = \mathbb{D}_q$ is a non-negative irreducible matrix. If $-1 < q <0$, then $1+\frac{1}{q} <0$. This shows the corollary.
\end{proof}
We now show that $q=-1$ only occurs when the eigenvalue is $-k$ and hence the distance-regular graph is bipartite.
\begin{lemma}
Let $\Gamma$ be a distance-regular graph with diameter $D$ at least 3 and valency $k$.
Let $\theta \neq k$ be an eigenvalue of $\Gamma$.
Then, $\theta$ is of classical $-1$-type if and only if $\theta= -k$ and $\Gamma$ is bipartite.
\end{lemma}
\begin{proof}
Let $\theta$ be of classical $-1$-type. Then, there exists a real number $c$ such that the standard sequence $(u_i)_i$ for $\theta$ satisfies $u_i = -u_{i-1} +c$.
This means that $u_1 = -1 +c$ and $u_2 = 1-c +c =1$. As $D \geq 3$, this means that $\Gamma$ is bipartite and $\theta= -k$, by \cite[Proposition 4.4.7]{BCN}.
\end{proof}
We wonder if a distance-regular graph has an eigenvalue of classical $q$-type, 
then the distance space of the  corresponding matrix $\mathbb{D}_q$ would satisfy 
the triangle inequality. 

Now we show that distance-regular graphs with classical parameters have eigenvalues of classical type.
\begin{lemma}\label{3ev2}
Let $\Gamma$ be a distance-regular graph with classical parameters $(D, b, \alpha, \beta)$ and the diameter $D \geq 3$. Then,
the eigenvalue $\theta= -1+\frac{b_1}{b}$ is of classical $b$-type.
In particular, $\mathbb{D}_b$ has exactly three distinct eigenvalues of which one is equal to 0.
\end{lemma}
\begin{proof}
Let $\Gamma$ be a distance-regular graph with classical parameters $(D, b, \alpha, \beta)$ and the diameter $D \geq 3$.
By \cite[Theorem 8.4.1 and Proposition 4.1.8]{BCN}, we find that there exists an eigenvalue $\theta\neq k$
whose standard sequence $(u_i)_i$ satisfies
$u_i = \frac{1}{b} u_{i-1} +c$ for $i =1,2, \ldots, D$, for some real number $c$. This means that, $\theta$ is of classical $b$-type. It follows by \cite[Corollary 8.4.2]{BCN} that
$\theta = -1 + \frac{b_1}{b}$. The ``in particular'' part follows from  Corollary \ref{3ev}.
\end{proof}
As a consequence of this lemma, we have the following result, where the distance eigenvalues
are the eigenvalues of the distance matrix $\mathbb{D} = \mathbb{D}_1$.
\begin{theorem}\label{clas1}
Let $\Gamma$ be one of the following distance regular graphs with diameter $D \geqslant 3$ 
\begin{enumerate}
\item[(i)] a Hamming graph or a Doob graph,
\item[(ii)] a Johnson graph,
\item[(iii)] a halved cube,
\item[(iv)] the Gosset graph.
\end{enumerate}
Then, $\Gamma$ has exactly $3$ distinct distance eigenvalues one of which is $0$.
\end{theorem}
\begin{proof}
Each of the graphs in the theorem is classical with parameters $(D, 1, \alpha, \beta)$  for  some
$D,~\alpha$ and $\beta$, see \cite[Table 6.1]{BCN}.
Now the theorem  follows from Lemma \ref{3ev2}.
\end{proof}

Later, we will give another proof of this result.

\begin{remark} In Aalipour et al. \cite{AAB2016}, they determined the distance spectra of the graphs in the above theorem. Those spectra can be easily derived from the above result, as the spectral distance radius is equal to $\sum_{i=1}^D ik_i$  and the multiplicity of the eigenvalue $b_1 -1$ is well-known. As the trace of $\mathbb{D}$ is equal to 0 we obtain easily the distance spectra of the graphs in the above theorem.
\end{remark}

\begin{lemma}
Let $\Gamma$ be a distance-regular graph with diameter $D$ at least 2 and distinct eigenvalues $\theta_0= k > \theta_1 > \cdots > \theta_D$.
Assume $\theta \neq k$ is an eigenvalue of classical $q$-type for some real number $q \neq 0$.
Then, $\theta = \theta_1$ if and only if $q $ is positive.
\end{lemma}
\begin{proof}
Let $(u_i)_i$ be the standard sequence for $\theta$. Let $\sigma_i := \frac{1-u_i}{1-u_1}$ for $i=0, 1, \ldots, D$.
Then, $\sigma_0 = 0$, $\sigma_1 = 1$ and $\sigma_i = \frac{1}{q}\sigma_{i-1} +1$ for $i =1,2, \ldots, D$ and
$\sigma_i = 1 +\frac{1}{q} + \cdots + \frac{1}{q^{i-1}}$ for $i =1,2, \ldots, D$. As $q >0$, we see that
$\sigma_i > \sigma_{i-1}$ for $i=1, 2, \ldots, D$. This means that $u_i < u_{i-1}$ for $i=1, 2, \ldots D$.
As $u_0 =1$ and $\theta \neq k$, it follows by \cite[Lemma 4.2 and Theorem 4.3]{XTLK2023} that $\theta= \theta_1$.
It also shows that, if $\theta_1$ is an eigenvalue of classical $q$-type, then $q$ is positive. This finishes the proof of the lemma.
\end{proof}

Now, we give a slight generalization of \cite[Proposition 4.4.9 (i)]{BCN}.
\begin{proposition}\label{quad}
 Let $\Ga$ be a distance-regular graph with intersection array\\ $\{k, b_1, \ldots, b_{D-1}; 1, c_2, \ldots, c_D\}$,
and with distinct eigenvalues $k > \theta_1 > \cdots > \theta_D$.
If $\Ga$ contains an induced $K_{r, r}$ for some integer $r \geq 2$, then $u_0 +(r-1)u_2 \geq ru_1$ or, equivalently,  $\theta_1 \leq \frac{b_1}{r-1} -1$ holds.
If you consider the standard representation $\sigma$ of $\Gamma$ with  respect to $\theta_1$ and $\theta_1 = \frac{b_1}{r-1}$, 
then the images of $\sigma$ for each of the two color classes have equal sum.
\end{proposition}
\begin{proof}

Let $x_1, x_2, \ldots x_r$ and $y_1, y_2, \ldots, y_r$ be the respective color classes of an induced $K_{r, r}$ of $\Gamma$, that is, $x_i \sim y_j$ for all $i, j$.
Let $(u_i)_i$ be the standard sequence with respect to $\theta_1$.
Let $\sigma$ be the standard representation with respect to $\theta_1$.
Consider the inner product $$p= \left(\sum_{i=1}^r\sigma(x_i) - \sum_{i=1}^r\sigma(y_i), \sum_{i=1}^r\sigma(x_i) - \sum_{i=1}^r\sigma(y_i)\right).$$  Then, $p = 2r(u_0 + (r-1)u_2 -ru_1)\geq 0$.
This means that $u_0 +(r-1)u_2 \geq ru_1$. This is equivalent with $\theta_1 \leq \frac{b_1}{r-1} -1$, and if $\theta_1 = \frac{b_1}{r-1} -1$, then $p=0$ and hence
$$\sum_{i=1}^r\sigma(x_i)  = \sum_{i=1}^r\sigma(y_i).$$
This completes the proof.

\end{proof}
As a consequence of this proposition, we obtain the following result.
\begin{corollary}(\cite[Proposition 4.4.9 (i)]{BCN})
 Let $\Ga$ be a distance-regular graph with intersection array $\{k, b_1, \ldots, b_{D-1}; 1, c_2, \ldots, c_D\}$,
and with distinct eigenvalues $k > \theta_1 > \cdots > \theta_D$.
If $\Ga$ contains an induced quadrangle and  $\theta_1 = b_1 -1$ holds,  then the standard sequence $(u_i)_i$ with respect to $\theta_1$ satisfies
$u_i = 1-i\left(\frac{a_1+2}{k}\right)$ holds for $i =0,1, \ldots, D,$ and hence $\theta_1$ is of classical 1-type.
\end{corollary}
\begin{proof}
Let $x\sim u \sim y \sim v \sim x$ be an induced quadrangle.
Let $z$ be a vertex such that $d(x, z) = i-1$ and $d(y, z) = i+1$ and hence $d(u, z) = d(w, z) = i$ for some $i= 1,2, \ldots, D-1$.
 Assume $\theta_1 = b_1 -1$.
Then, by Proposition \ref{quad}, we have  $$ u_{i-1} + u_{i+1} = \left(\sigma(z), \sigma(x) + \sigma(y)\right) = (\sigma(z), \sigma(u) + \sigma(w)) = 2u_i.$$
So, we find $u_{i-1} + u_{i+1} = 2 u_i$ for $i = 1, 2, \ldots, D-1$. By an easy induction one finds that $u_i = 1-i\frac{a_1+2}{k}$. This implies that $\theta_1$ is of classical 1-type.
\end{proof}
This corollary gives another proof for Theorem \ref{clas1}.

Now, we give another consequence of Proposition \ref{quad}.
\begin{proposition}
Let $\Ga$ be a distance-regular graph with intersection array\\ $\{k, b_1, \ldots, b_{D-1}; 1, c_2, \ldots, c_D\}$,
and with distinct eigenvalues $k > \theta_1 > \cdots > \theta_D$. Let $\theta_1$ be an eigenvalue of classical $q$-type.
If $\Ga$ contains an induced $K_{r, r}$ for some integer $r \geq 2$, then $q \geq r-1$.
\end{proposition}
\begin{proof}
Let $(u_i)_i$ be the standard sequence for $\theta_1$.
By Proposition \ref{quad}, we have $u_0 + (r-1)u_2 \geq ru_1$.
As $\theta_1$ is of classical $q$-type, we have $q >0$ and $u_i = \frac{1}{q} u_{i-1} +c$ where $c = u_1 -\frac{1}{q}$, because $u_0=1$. It follows that $u_2 = u_1 +\frac{1}{q}\left(u_1-1\right)$.
We have $\frac{r-1}{q}\left(u_1-1\right)=(r-1)(u_2 -u_1) \geq u_1 - u_0 = u_1 -1$. As $u_1 <1$ as $\theta_1 \neq k$, we find 
$r-1 \leq q$. This shows the proposition.

\end{proof}

\section{Distance-regular graph with exactly one positive $q$-distance eigenvalue}
In this section, we study distance-regular graphs with exactly one positive $q$-distance eigenvalue where $q \neq 0$.
We show some properties of any local graph of such a distance-regular graph.

\begin{proposition}\label{regular}
Let $\Gamma$ be a distance-regular graph with diameter $D \geq2$ and distinct eigenvalues $k = \theta_0 > \theta_1 > \cdots > \theta_D$.
Let $\Delta$ be an induced regular subgraph of $\Gamma$ (with at least two vertices) such that any two vertices in $\Delta$ have distance at most two in $\Gamma$.
\begin{enumerate}
\item[(i)] Assume that the $q$-distance matrix $\mathbb{D}_q$ of $\Gamma$ has exactly one positive eigenvalue for some $q >0$. Then, the smallest eigenvalue $\lambda_{\min}$ of $\Delta$ satisfies
$\lambda_{\min} \geq -q -1$;
\item[(ii)] Assume that the $q$-distance matrix $\mathbb{D}_q$ of $\Gamma$ has exactly one positive eigenvalue for some $q <0$. Then, the second largest eigenvalue $\lambda_{2}$ of $\Delta$ satisfies
$\lambda_{2} \geq -q -1$.

\end{enumerate}
\end{proposition}
\begin{proof}
Let $\widehat{\mathbb{D}}$ be the principal submatrix of $\mathbb{D}_q$ indexed by $V(\Delta)$. Let $B$ be the adjacency matrix of $\Delta$.
Then, $\widehat{\mathbb{D}}$ satisfies $\widehat{\mathbb{D}} = B +\left(1+ \frac{1}{q}\right)(J-I - B)$. By interlacing, we have that the second largest eigenvalue of $\widehat{\mathbb{D}}$ is at most 0. Also note that
$\widehat{\mathbb{D}}\mathbf{j} = \eta \mathbf{j}$ where $\eta >0$ and $\mathbf{j}$ is the all-ones vector. Let $\mathbf{v}$ be an eigenvector of $\widehat{\mathbb{D}}$, say with respect to eigenvalue $\eta_1$, orthogonal to the all-ones vector. Then, $\eta_1 \leq 0$ and $\eta_1 \mathbf{v} = \widehat{\mathbb{D}}\mathbf{v}= \left(-\frac{1}{q}B - \frac{q+1}{q}I\right) \mathbf{v}$. 
Let $q >0$. Then this means that
$\mathbf{v}$ is an eigenvector of $B$ with eigenvalue $-q\eta_1 - (q+1) \geq -q-1$. This shows the first item of the proposition.
The proof for the second item is the same as the proof for the first item.
\end{proof}
Recall that the local graph $\Delta(x)$ at a vertex $x$ of a graph $\Gamma$ is the subgraph induced on the neighbours of $x$ in $\Gamma$.

We have the following immediate consequence of Proposition \ref{regular}.
\begin{theorem}
Let $\Gamma$ be a distance-regular graph with diameter $D \geq2$, valency $k \geq 2$ and distinct eigenvalues $k = \theta_0 > \theta_1 > \cdots > \theta_D$.
Let $x$ be a vertex of $\Gamma$.
\begin{enumerate}
\item[(i)] Assume that the $q$-distance matrix $\mathbb{D}_q$ has exactly one positive eigenvalue for some $q >0$. Then, the smallest eigenvalue $\lambda_{\min}$ of $\Delta(x)$ satisfies
$\lambda_{\min} \geq -q -1$;
\item[(ii)] Assume that the $q$-distance matrix $\mathbb{D}_q$ has exactly one positive eigenvalue for some $q <0$. Then, the second largest eigenvalue $\lambda_{2}$ of $\Delta(x)$ satisfies
$\lambda_{2} \geq -q -1$.
\end{enumerate}
\end{theorem}
\begin{proof}
It follows immediately from Proposition \ref{regular}, as $\Delta(x)$ is $a_1$-regular.
\end{proof}

Now we will show that if the $q$-distance matrix has exactly one positive eigenvalue, then the $q$-distance matrix has only non-negative entries.
\begin{lemma}\label{pos}
Let $\Gamma$ be a distance-regular graph with $n$ vertices, diameter $D \geq2$ and distinct eigenvalues $k = \theta_0 > \theta_1 > \cdots > \theta_D$.
Assume that the $q$-distance matrix $\mathbb{D}_q$ has exactly one positive eigenvalue for some $q \neq0$. Then, $\mathbb{D}_q$ has only non-negative entries and hence 
the all-one-vector is an eigenvector with respect to the positive eigenvalue. In particular, we see that $q>0$ or $q \leq -1$ holds.
\end{lemma}
\begin{proof}
Let $E$ be a primitive idempotent of $\Gamma$. Then $\mathbb{D}_q E = \eta E$ for some real number $\eta$. So $\eta$ has multiplicity at least Rk$(E)$. So if $\eta$ has multiplicity $1$, then Rk$(E) =1$. By \cite[Proposition 4.4.8]{BCN}, we see that $E = \frac{1}{n} J$ or $AE = -kE$ and $\Gamma$ is bipartite. 
Let $E$ be such that $AE = -kE$ and assume that $\mathbb{D}_q E = \eta E$ for some  $\eta >0$. This means that $0< \eta = \sum_{i=1}^D (-1)^i k_i(1+\frac{1}{q} +\cdots + \frac{1}{q^{i-1}})$. We also have $\mathbb{D}_q \mathbf{j} = \lambda \mathbf{j}$ for some $\lambda \leq 0$. It follows that $0\geq \lambda = \sum_{i=1}^D  k_i(1+\frac{1}{q} +\cdots + \frac{1}{q^{i-1}})$. Let $t$ be the largest integer such that $2t+1 \leq D$. Hence $0 > \lambda -\eta = \sum_{j=0}^t 2k_{2j+1}(1+\frac{1}{q} +\cdots + \frac{1}{q^{2j}})>0$, a contradiction.
This means that this case can not happen and hence $E = \frac{1}{n}J$. Let $\mathbf{e}_x$ be the the unit vector such that its $x$-entry is equal to 1 and the rest to 0.
Then, for distinct vertices $x$ and $y$, $(\mathbf{e}_x- \mathbf{e}_y)^T\mathbb{D}_q (\mathbf{e}_x- \mathbf{e}_y) \leq 0$, which is equivalent with $(\mathbb{D}_q)_{xy} \geq 0$. 
By Lemma \ref{nonnegative}, we obtain the ``in particular'' part. This finishes the proof of the lemma.
\end{proof}

From Proposition \ref{regular} and Lemma \ref{pos}, we have the following.
\begin{theorem}
Let $\Gamma$ be a distance-regular graph with diameter $D \geq2$ and distinct eigenvalues $k = \theta_0 > \theta_1 > \cdots > \theta_D$.
Assume that the $q$-distance matrix $\mathbb{D}_q$ has exactly one positive eigenvalue for some $-1< q< 1$.
Then, $c_2 =1$ and $0 < q < 1$. 
\end{theorem}
\begin{proof}
By Lemma \ref{pos}, we obtain that $1 > q >0$. 
We need to consider three cases: 
\begin{enumerate}
\item[(i)] $\Gamma$ has an induced $4$-gon;
\item[(ii)] $\Gamma$ has no induced $4$-gon, but $c_2 \geq 2$;
\item[(iii)] $c_2 =1$.
\end{enumerate}

Let $Q$ be an induced $4$-gon in $\Gamma$. Then, the $q$-distance matrix of $Q$, i.e. $\mathbb{D}_q(Q)$, 
has exactly one positive eigenvalue if and only if $q \geq 1$ by Proposition 
\ref{regular}, so this is not possible.

By interlacing, we obtain that the $q$-distance matrix of $\Gamma$ has only one positive eigenvalue only if $q\geq 1$ or $ q \leq -1$.

Now assume that $c_2 \geq 2$, but that $\Gamma$ has no induced $4$-gons. In this case, we have that $\Delta(x)$ is the 
$s$-clique-extension of a connected non-complete strongly-regular graph with no induced 
$4$-gons for some $s \geq 1$, by \cite[Theorem 1.16.3]{BCN}.
The only connected non-complete strongly-regular graph with smallest eigenvalue less than $-2$ is the pentagon.  
In this case, we have that $\Delta(x)$ is the pentagon and $\Gamma$ is the icosahedron, by \cite[p. 35]{BCN}. 
It is easily checked that the $q$-distance matrix of the icosahedron has exactly one positive eigenvalue 
if and only if $q \geq 1$. This shows the theorem.
\end{proof}
\begin{remark}
(i) Note that the $q$-distance matrix of an odd polygon has only one positive eigenvalue for some $0 < q <1$. 
We wonder whether they are the only ones. \\
(ii) Koolen and Shpectorov \cite{KS1994} classified the distance-regular graphs with exactly one positive distance eigenvalue.
\end{remark}

\section*{Declarations}

\subsection*{Declaration of Competing Interest}
The authors declare that there are no competing interests.

\subsection*{Acknowledgements}
J.H. Koolen is partially supported by the National Key R. and D. Program of China (No. 2020YFA0713100), 
the National Natural Science Foundation of China (No. 12071454), and the Anhui Initiative in Quantum 
Information Technologies (No. AHY150000). M. Abdullah is supported by the Chinese Scholarship Council at USTC, China.
B. Gebremichel is supported by the National Key R. and D. Program of China (No. 2020YFA0713100) and 
the Foreign Young Talents Program (No. QN2022200003L).
S. Hayat is supported by UBD Faculty Research Grants 
(No. UBD/RSCH/1.4/FICBF(b)/2022/053).

\subsection*{Data availability}
No data was used for the research described in the article.



\end{document}